\documentclass[reqno,11pt]{amsart}
\usepackage{amssymb,amsmath,hyperref,enumerate}

\numberwithin{equation}{section}
\newtheorem{theorem}[equation]{Theorem}
\newtheorem{lemma}[equation]{Lemma}
\newtheorem{proposition}[equation]{Proposition}

\theoremstyle{definition}
\newtheorem{example}[equation]{Example}
\newtheorem*{remark}{Remark}

\def\Q{\mathbb{Q}}
\def\Z{\mathbb{Z}}

\providecommand{\norm}[1]{\lVert#1\rVert}

\title{On Hensel's roots and a factorization formula in $\Z[[x]]$}
\author{Daniel Birmajer}
\address{Department of Mathematics\\ Nazareth College\\ 4245 East Ave.\\ Rochester, NY 14618}
\author{Juan B. Gil}
\address{Penn State Altoona\\ 3000 Ivyside Park\\ Altoona, PA 16601}
\author{Michael D. Weiner}


\begin{document}
\maketitle

\begin{abstract}
Given an odd prime $p$, we provide formulas for the Hensel lifts of polynomial roots modulo $p$, and give an explicit factorization over the ring of formal power series with integer coefficients for certain reducible polynomials whose constant term is of the form $p^w$ with $w>1$. All of our formulas are given in terms of partial Bell polynomials and rely on the inversion formula of Lagrange.
\end{abstract}

\section{Introduction}\label{sec:introduction}

The divisibility theory of commutative rings is a fundamental and persisting topic in mathematics that entails two main aspects: determining irreducibility and finding a factorization of the reducible elements in the ring. Prominent examples are the ring of integers $\Z$ and the ring of polynomials $\Z[x]$. It is then natural to investigate the arithmetic properties of $\Z[[x]]$, the ring of formal power series with integer coefficients. While polynomials in $\Z[x]$ can be seen as  power series over the integers, the factorization properties over $\Z[x]$ and over $\Z[[x]]$ are in general unrelated; cf. \cite{BiGi}. In \cite{BGW2012a}, the authors studied this factorization problem exhaustively. In particular, for a class of polynomials parametrized by a prime $p$, a connection between reducibility in $\Z[[x]]$ and the existence of a $p$-adic root with positive valuation was established. Whereas this connection can be certainly explained in structural terms, the role of the root in the factorization process is not obvious.

Motivated by this factorization problem and the need to find explicit $p$-adic roots, the main goal of this paper is to provide formulas for the Hensel lifts of roots modulo $p$, and to give a factorization in $\Z[[x]]$ of certain reducible polynomials whose constant term is of the form $p^w$ with $p$ prime and $w>1$. All of our formulas are given in terms of partial Bell polynomials and rely on the inversion formula of Lagrange.

On the one hand, in Section~\ref{sec:Hensel}, we prove two versions of Hensel's lemma that give explicit formulas for the roots of any polynomial in $\Z_p[x]$, the ring of polynomials over the $p$-adic integers $\Z_p$. For illustration purposes, we examine the special cases of quadratic and cubic polynomials, see \eqref{quadraticRoot} and \eqref{cubicRoot}, and discuss the roots of unity leading to a formula for the so-called Teichm\"uller lifts, see Proposition~\ref{roots_unity}. On the other hand, we give a factorization over $\Z[[x]]$ for polynomials $f$ (of degree higher than 1) with $f(0)=p^w$ that are reducible in the presence of a $p$-adic root in $p\Z_p$. Although Theorem~\ref{FactorizationThm} is formulated for polynomials, it actually holds verbatim for power series. An illustrative example is discussed at the end of Section~\ref{sec:Factorization}. 

As mentioned before, sections~\ref{sec:Hensel} and \ref{sec:Factorization} are related and rely on the material discussed in Section~\ref{sec:Inversion}. For the reader's convenience, a short appendix with some of the basic properties and identities for the partial Bell polynomials is included. We finish by observing that most of the results presented here may be applied to polynomials and power series over other commutative rings.

\section{Series solutions of algebraic equations}\label{sec:Inversion}

The main results of this paper rely on the following consequence of the inversion formula of Lagrange for formal power series. For a detailed proof and other applications, we refer the reader to \cite[Section 3.8]{Comtet} or \cite[Section 11.6]{Charalambides}. In the sequel, $B_{n,j}(x_1,x_2,\dots)$ denotes the $(n,j)$-th partial Bell polynomial, see the appendix.

\begin{lemma}[cf. {\cite[Corollary 11.3]{Charalambides}}]
\label{FormalInversion}
If $\phi(t)$ is a power series of the form
\begin{equation*}
\phi(t) = t\bigg(1+\sum_{r=1}^\infty \alpha_r \frac{t^r}{r!}\bigg),
\end{equation*}
then its formal inverse is given by
\begin{equation*}
\phi^{-1}(u) = u\bigg(1+\sum_{n=1}^\infty \beta_n \frac{u^n}{n!}\bigg),
\end{equation*}
where
\begin{equation*}
 \beta_n = \sum_{j=1}^n (-1)^j \frac{(n+j)!}{(n+1)!}B_{n,j}(\alpha_1,\alpha_2,\dots).
\end{equation*}
\end{lemma}

Inversion formulas of this type have been studied by many authors in the search for solutions of algebraic equations. For instance, a series solution for the equation $x^m+px=q$ was already given by Lambert in 1758, cf. \cite{Lambert}. The most general formulas we found in the literature were obtained by Birkeland around 1927.  In \cite{Birk27}, the author studied arbitrary polynomial equations and obtained explicit solutions in terms of hypergeometric functions, see also \cite{PaTsi}. 

It turns out that, if $f(x)$ is a power series over a commutative ring $\mathcal{R}$ with an invertible linear coefficient, then formal series solutions for the equation $f(x)=0$ can be obtained from Lemma~\ref{FormalInversion} as follows. 

\begin{proposition}\label{RootThm}
Given a power series $f(x) = a_0 + a_1x + a_2 x^2 + \cdots\in \mathcal{R}[[x]]$ with $a_1$ invertible in $\mathcal{R}$, the equation $f(x)=0$ has the formal root
\begin{equation}\label{formalRoot}
 x = \sum_{n=0}^\infty \sum_{k=0}^n \frac{(-1)^{n-k+1}}{a_1^k\,(n+1)!} \binom{2n+1}{n-k} B_{n+k,k}(1! a_1, 2!a_2,\dots) \Big(\frac{a_0}{a_1}\Big)^{n+1}.
\end{equation}
\end{proposition}
\begin{proof}
Let 
\begin{equation*}
\phi(x) = x\bigg(1+\sum_{\ell=1}^\infty \alpha_\ell \frac{x^\ell}{\ell!}\bigg) \text{ with } \alpha_\ell=\ell!\, a_{\ell+1}/a_1.
\end{equation*}
Thus $f(x)=a_1\big(a_0/a_1+\phi(x)\big)$ and $f(x)=0$ if $\phi(x)=-a_0/a_1$. By Lemma~\ref{FormalInversion}, this equation has the formal root
\begin{align} \notag
x&= \phi^{-1}(-a_0/a_1) = -\frac{a_0}{a_1}\bigg(1+\sum_{n=1}^\infty \frac{(-1)^n\beta_n}{n!} \Big(\frac{a_0}{a_1}\Big)^n \bigg) \\ \notag
 &= -\frac{a_0}{a_1}\bigg(1+\sum_{n=1}^\infty \frac{(-1)^{n}}{n!} \bigg[\sum_{j=1}^n  (-1)^j\frac{(n+j)!}{(n+1)!}B_{n,j}(\alpha_1,\alpha_2,\dots)\bigg] \Big(\frac{a_0}{a_1}\Big)^n \bigg) \\
\label{altForm} 
 &= \sum_{n=0}^\infty \bigg[\sum_{j=0}^n \frac{(-1)^{n+j+1}}{a_1^j\, n!}  \frac{(n+j)!}{(n+1)!}B_{n,j}(1!a_2, 2!a_3,\dots)\bigg] \Big(\frac{a_0}{a_1}\Big)^{n+1}.
\end{align}
Now, if we set $x_j=j! a_j$, then
\begin{equation*}
 B_{n,j}(1! a_2, 2!a_3,\dots) 
 = B_{n,j}(\tfrac{x_2}{2} \tfrac{x_3}{3},\dots) = \frac{n!}{(n+j)!} B_{n+j,j}(0,x_2,x_3,\dots)
\end{equation*}
by \eqref{Comt3l'}. Moreover, by means of \eqref{Comt3n} and \eqref{Comt3n'}, we have  
\begin{align*}
 B_{n+j,j}(0,x_2,x_3,\dots) &= \sum_{\substack{k\le j \\ \nu\le n+j}} 
\binom{n+j}{\nu}B_{\nu,k}(x_1,x_2,\dots) B_{n+j-\nu,j-k}(-x_1,0,\dots) \\
 &= \sum_{k\le j} \binom{n+j}{n+k} (-x_1)^{j-k}B_{n+k,k}(x_1,x_2,\dots).
\end{align*}
Therefore,
\begin{align*} 
\sum_{j=0}^n \frac{(-1)^{n+j+1}}{a_1^j\,n!} & \frac{(n+j)!}{(n+1)!} B_{n,j}(1!a_2, 2!a_3,\dots) \\ 
&= \sum_{j=0}^n \frac{(-1)^{n+j+1}}{a_1^j\, (n+1)!} B_{n+j,j}(0,2!a_2,3!a_3,\dots) \\ 
&= \sum_{j=0}^n\sum_{k\le j} \frac{(-1)^{n-k+1}}{a_1^k\, (n+1)!} \binom{n+j}{n+k}
B_{n+k,k}(1!a_1,2!a_2,\dots) \\[-3ex] \intertext{}
&= \sum_{k=0}^n \frac{(-1)^{n-k+1}}{a_1^k\, (n+1)!} \bigg[\sum_{j=k}^n \binom{n+j}{n+k}\bigg] 
B_{n+k,k}(1!a_1,2!a_2,\dots) \\ 
&= \sum_{k=0}^n \frac{(-1)^{n-k+1}}{a_1^k\, (n+1)!} \binom{2n+1}{n-k} B_{n+k,k}(1!a_1,2!a_2,\dots). 
\end{align*}
Inserting this expression into \eqref{altForm}, we arrive at \eqref{formalRoot}.
\end{proof}

The simplicity (or complexity) of formula \eqref{formalRoot} clearly depends on the structure of the partial Bell polynomials. 

For example, for $x^m+px-q=0$ with $p,q\in\mathbb{R}$, $p\not=0$, $m>1$, the root \eqref{formalRoot} takes the form
\begin{equation*}
 x = \sum_{n=0}^\infty \sum_{k=0}^n \frac{(-1)^{n-k+1}}{p^k\,(n+1)!} \binom{2n+1}{n-k} 
 B_{n+k,k}(p,0,\dots,0,m!,0,\dots) \Big(\frac{-q}{p}\Big)^{n+1},
\end{equation*}
which by means of \eqref{Comt3n} and \eqref{Comt3n'} reduces to
\begin{equation}\label{LambertRoot}
 x = \sum_{k=0}^\infty \frac{(-1)^{k}}{p^{k}} \binom{mk}{k} \frac{1}{(m-1)k+1} 
 \Big(\frac{q}{p}\Big)^{(m-1)k+1}.
\end{equation} 
This formula includes Eisenstein's series solution for $x^5+x=q$, cf. \cite{Stillwell},
\begin{equation*}
 x = \sum_{k=0}^\infty (-1)^{k} \binom{5k}{k} \frac{1}{4k+1} q^{4k+1}.
\end{equation*}

Of course, this series does not converge for all values of $q$, so further analysis is required to understand and possibly make sense of \eqref{formalRoot}. While convergence in general is not the focus of this paper, we want to briefly discuss $f(x)=x^3+x-q$ in order to illustrate a possible analytic approach. For this polynomial, the sum \eqref{LambertRoot} becomes 
\begin{equation*}
 x = \sum_{k=0}^\infty (-1)^{k} \binom{3k}{k} \frac{1}{2k+1} q^{2k+1}
 = q\sum_{k=0}^\infty \binom{3k}{k} \frac{1}{2k+1} (-q^2)^{k},
\end{equation*}
which converges only when $q^2\leq 4/27$. However, 
\begin{equation*}
 \sum_{k=0}^\infty \binom{3k}{k} \frac{1}{2k+1} (-q^2)^{k}
 = {}_{{}_2}F_{{}_1}\big(\tfrac13, \tfrac23; \tfrac32; -\tfrac{27}{4} q^2\big),
\end{equation*}
and since the hypergeometric function ${}_{{}_2}F_{{}_1}\big(\tfrac13, \tfrac23; \tfrac32; z\big)$ extends analytically to $\mathbb{C}\setminus (1,\infty)$, we can actually evaluate the formal root for larger values of $q$. For example, if $q=2$, then the root of $x^3+x-2=0$ provided by \eqref{formalRoot} is precisely $x=2\cdot {}_{{}_2}F_{{}_1}\big(\tfrac13, \tfrac23; \tfrac32; -27\big)=1$. 

\bigskip
In the next section we will fully discuss the use of \eqref{formalRoot} to find roots of polynomials over the $p$-adic integers. 

\section{Hensel's roots}\label{sec:Hensel}

In this section, we use Proposition~\ref{RootThm} to give a version of Hensel's lemma that provides an explicit formula for the $p$-adic root of a polynomial in $\Z_p[x]$. We start by recalling some basic facts about the $p$-adic numbers. For a comprehensive treatment of this subject, the reader is referred to \cite{Katok,Koblitz,Serre}.
\par
Let $p$ be a prime integer. For any nonzero integer $a$, let $v_p(a)$ (the $p$-adic valuation of $a$) be the highest power of $p$ which divides $a$, i.e., the greatest $m$ such that $a \equiv 0 \pmod {p^m}$; we agree to write $v_p(0) = \infty$.  Note that  $v_p\,(a_1\,a_2) = v_p(a_1) + v_p(a_2)$ for all $a_1, a_2\in \Z$. For any rational number $x = a/b$, define $v_p(x)=v_p(a) - v_p(b)$. Note that this expression depends only on $x$ and not on its representation as a ratio of integers. 

The $p$-adic norm in $\Q$ is defined  as $\norm{x}_p=p^{-v_p(x)}$ if $x\not=0$, and $\norm{0}_p=0$. This norm is non-Archimedean, that is  $\norm{x + y}_p \le  \max(\norm{x}_p, \norm{y}_p)$. The $p$-adic completion of $\Q$ with respect to $\norm{\cdot}_p$ is denoted by $\Q_p$. Every $a\in \Q_p$  admits a unique $p$-adic expansion,
\begin{equation*}
a = \frac{a_0}{p^m}+\frac{a_1}{p^{m-1}}+\dotsb +\frac{a_{m-1}}{p}+a_m + a_{m+1}\,p + a_{m+2}\,p^2+\dotsb,  
\end{equation*}
with $0\le a_i<p$ for all $i$.
\par
We let $\Z_p = \{a \in  \Q_p\,\vert\, \norm{a}_p \le 1\}$, the set of all numbers in $\Q_p$ whose $p$-adic expansion involves no negative powers of $p$. An element of $\Z_p$ is called a \emph{p-adic integer}, and the set of $p$-adic integers is  a subring of the field $\Q_p$. If $x\in \Z_p$ is such that $v_p(x)=0$, then $x$ is a unit and its multiplicative inverse $1/x$ is in $\Z_p$.
\par
A fundamental property of the $p$-adic numbers is that a series in $\Q_p$ converges if and only if its terms approach zero. This condition is equivalent to verifying that the $p$-adic valuation of the terms tend to infinity.

\begin{theorem}\label{basicHensel}
Let $p>2$ be prime and let $f(x) = a_0 + a_1x + \cdots + a_m x^m$ be a polynomial of degree $m$ in $\Z_p[x]$. If $r_0\in\Z$ is such that
\begin{equation*}
 f(r_0)\equiv 0\!\! \pmod p \;\text{ and }\; v_p(f'(r_0))=0,
\end{equation*}
then $r_0$ lifts to a $p$-adic root $r$ of $f$ given by
\begin{equation}\label{HenselRoot}
 r = r_0 + \sum_{n=0}^\infty \bigg[\sum_{k=0}^n \frac{(-1)^{n-k+1}}{c_1^k\,(n+1)!} \binom{2n+1}{n-k} 
 B_{n+k,k}(1!c_1, 2!c_2,\dots)\bigg] \Big(\frac{c_0}{c_1}\Big)^{n+1},
\end{equation}
where $c_j=\frac{f^{(j)}(r_0)}{j!}$ for $j=0,1,\dots,m$. Note that $v_p(c_1)=0$ implies $1/c_1\in \Z_p$.
\end{theorem}
\begin{proof}
Given $f(x)$ and $r_0\in\Z$ as above, consider the function $g(x) = f(r_0+x)$. The Taylor expansion of $g(x)$ at $x=0$ gives
\begin{align*}
 g(x) = f(r_0) + f'(r_0) x + \tfrac{f''(r_0)}{2!} x^2 +\dots +\tfrac{f^{(m)}(r_0)}{m!} x^m 
 = c_0 + c_1 x + c_2 x^2 +\dots + c_m x^m,
\end{align*}
with the property that $v_p(c_0)=v_p(f(r_0))\ge 1$ and $v_p(c_1)=v_p(f'(r_0))=0$. Thus $c_1\not=0$, and by Proposition~\ref{RootThm}, $g(x)$ has a formal root 
\begin{equation*}
 \varrho = \sum_{n=0}^\infty \gamma_n \frac{c_0^{n+1}}{(n+1)!},
\end{equation*}
where
\begin{equation*}
  \gamma_n = \sum_{k=0}^n \frac{(-1)^{n-k+1}}{c_1^{k+n+1}} \binom{2n+1}{n-k} 
  B_{n+k,k}(1!c_1, 2!c_2,\dots).
\end{equation*}
Since $v_p(c_1)=0$ and each $j!c_j$ is a $p$-adic integer, we have $\gamma_n \in\Z_p$ for every $n$. Moreover, if $n+1$ has the $p$-adic expansion $n+1=n_0+n_1p+n_2p^2+\cdots$, we have
\begin{equation*}
 v_p((n+1)!)=\frac{n+1-s_p(n+1)}{p-1} < \frac{n+1}{p-1},
\end{equation*}
where $s_p(n+1)=n_0+n_1+n_2+\cdots$. Therefore, since $v_p(c_0)\ge 1$, we get
\begin{equation*}
 v_p\big(\tfrac{c_0^{n+1}}{(n+1)!}\big)=v_p(c_0^{n+1}) - v_p((n+1)!)>n+1-\frac{n+1}{p-1} = \big(\tfrac{p-2}{p-1}\big)(n+1)\to \infty \text{ as } n\to\infty,
\end{equation*}
which implies that $\sum \gamma_n\, \frac{c_0^{n+1}}{(n+1)!}$ converges in $p\Z_p$. In conclusion, the formal root $\varrho$ is indeed a $p$-adic root of $g(x)$ and $r_0+\varrho\in\Z_p$ is a root of $f(x)$. 
\end{proof}

More generally, we have:
\begin{theorem}\label{extendedHensel}
Let $p>0$ be prime and let $f(x) = a_0 + a_1x + \cdots + a_m x^m$ be a polynomial in $\Z_p[x]$. Let $\nu, \kappa\in\Z$ such that $0\le 2\kappa<\nu$. If $r_0\in\Z$ is such that
\begin{equation*}
 f(r_0)\equiv 0\!\! \pmod{p^\nu} \;\text{ and }\; v_p(f'(r_0))=\kappa,
\end{equation*}
then $r_0$ lifts to a $p$-adic root $r$ of $f$ given by
\begin{equation*}
 r = r_0 + p^\kappa\sum_{n=0}^\infty \bigg[\sum_{k=0}^n \frac{(-1)^{n-k+1}}{c_1^k\,(n+1)!} \binom{2n+1}{n-k} 
 B_{n+k,k}(1!c_1, 2!c_2,\dots)\bigg] \Big(\frac{c_0}{c_1}\Big)^{n+1},
\end{equation*}
where $c_j=p^{(j-2)\kappa}\,\frac{f^{(j)}(r_0)}{j!}$ for $j=0,1,\dots,m$.
\end{theorem}
\begin{proof}
The proof is similar to the one for the previous theorem. Let $r_0$ be a root of $f$ modulo $p^\nu$ and let $\kappa$ be the $p$-adic valuation of $f'(r_0)$. Consider $g(x)=p^{-2\kappa}f(r_0+p^\kappa x)$. A Taylor expansion of $g(x)$ at $0$ gives
\begin{align*}
 g(x) &= p^{-2\kappa}f(r_0) + p^{-\kappa}f'(r_0) x + \tfrac{f''(r_0)}{2!} x^2 +\dots 
 + p^{(m-2)\kappa}\,\tfrac{f^{(m)}(r_0)}{m!} x^m\\
 &= c_0 + c_1 x + c_2 x^2 +\dots + c_m x^m.
\end{align*}
If $0\le 2\kappa<\nu$, then $v_p(c_0)\ge 1$ and $v_p(c_1)=0$ since $v_p(f(r_0))\ge \nu>2\kappa$ and $v_p(f'(r_0))=\kappa$. At this point, we can proceed as in the proof of Theorem~\ref{basicHensel} and conclude that the formal root of $g(x)$ provided by \eqref{formalRoot} is indeed a $p$-adic root of $g(x)$. If we denote that root by $\varrho$, then $r=r_0+p^\kappa \varrho$ is a $p$-adic root of the polynomial $f(x)$. 
\end{proof}

In the case of quadratic and cubic polynomials, one can use known properties of Bell polynomials to give a  simpler representation of the corresponding Hensel's roots.

\subsection*{Quadratic Polynomials}
Let $p>2$ and $f(x)=a_0+a_1 x+a_2 x^2 \in \Z_p[x]$, $a_2\not=0$. If there is an $r_0\in\Z$ such that $f(r_0)\equiv 0$ (mod $p$) and $v_p(f'(r_0))=0$, then by Theorem~\ref{basicHensel} and elementary Bell polynomial identities, the $p$-adic lift of $r_0$ may be written as
\begin{equation} \label{quadraticRoot}
 r = r_0 - \frac{c_0}{c_1}\sum_{n=0}^\infty \binom{2n}{n}\frac{1}{n+1} \Big(\frac{c_0 c_2}{c_1^2}\Big)^n\in\Z_p,
\end{equation}
where $c_0=f(r_0)$, $c_1=f'(r_0)$, and $c_2=a_2$. Note that $\binom{2n}{n}\frac{1}{n+1}\in\Z$ are the well-known Catalan numbers.

\begin{example} 
Let us consider $f(x)=1+11x-5x^2$ over $\Z_7$. This polynomial has two simple roots mod 7, $r_0=1, 4$. Since $c_0=f(1)=7$, $c_1=f'(1)=1$, and $c_2=-5$, the lift of $r_0=1$ in $\Z_7$ is given by
\begin{equation*}
 r = 1 - \sum_{n=0}^\infty \binom{2n}{n}\frac{(-5)^n}{n+1}\, 7^{n+1}.
\end{equation*}
On the other hand, since $f(4)=-35$ and $f'(4)=-29$, the lift of $r_0=4$ in $\Z_7$ is given by
\begin{equation*}
 r = 4 - \sum_{n=0}^\infty \binom{2n}{n}\frac{1}{n+1} \Big(\frac{5}{29}\Big)^{2n+1} 7^{n+1}.
\end{equation*}
Note that $1/29=1+3\cdot 7+7^2+7^3+2\cdot 7^4+5\cdot 7^5+O(7^6)$ is an element of $\Z_7$.
\end{example}

\begin{example} 
We now consider $f(x)=17+6x+2x^2$ over $\Z_5$. Modulo $5$ this polynomial has a double root, $r_0=1$. Since $f(1)=5^2$ and $f'(1)=2\cdot 5$, we cannot apply any of the above theorems directly. However, the polynomial
\begin{equation*}
 g(x) = \tfrac{1}{25}f(1+5x) = 1+2x+2x^2
\end{equation*}
has $1$ and $3$ as simple roots modulo $5$, so using \eqref{quadraticRoot}, we get the lifts 
\begin{gather*}
 1-\frac{5}{6}\sum_{n=0}^\infty \binom{2n}{n}\frac{1}{n+1} \Big(\frac{5}{18}\Big)^n \;\text{ and }\;
 3-\frac{25}{14}\sum_{n=0}^\infty \binom{2n}{n}\frac{1}{n+1} \Big(\frac{25}{98}\Big)^n \;\text{ in } \Z_{5}.
\end{gather*}
Therefore, the $5$-adic roots of $f(x)$ are given by 
\begin{gather*}
 1+ 5-\frac{25}{6}\sum_{n=0}^\infty \binom{2n}{n}\frac{1}{n+1} \Big(\frac{5}{18}\Big)^n \;\text{ and }\;
 1+ 3\cdot 5 -\frac{125}{14}\sum_{n=0}^\infty \binom{2n}{n}\frac{1}{n+1} \Big(\frac{25}{98}\Big)^n.
\end{gather*}
\end{example}

\subsection*{Cubic Polynomials}
Let $p>2$ and $f(x)=a_0+a_1 x+a_2 x^2 + a_3 x^3\in \Z_p[x]$, $a_3\not=0$. Once again, if there is an $r_0\in\Z$ such that $f(r_0)\equiv 0$ (mod $p$) and $v_p(f'(r_0))=0$, then Theorem~\ref{basicHensel} gives a formula for the $p$-adic lift of $r_0$. However, for cubic polynomials, it is more convenient to use the equation \eqref{altForm} and write the root as
\begin{equation*}
  r=r_0+\sum_{n=0}^\infty \bigg[\sum_{k=0}^n \frac{(-1)^{n+k+1}}{c_1^k\, n!}  \frac{(n+k)!}{(n+1)!}B_{n,k}(c_2, 2c_3,0,\dots)\bigg] \Big(\frac{c_0}{c_1}\Big)^{n+1},
\end{equation*}
where $c_j=f^{(j)}(r_0)/j!$ for $j=0,1,2,3$. Note that $B_{n,k}(c_2,2c_3,0,\dots)=0$ for $k < n/2$, and for $k\ge n/2$, identities \eqref{Comt3n} and \eqref{Comt3n'} give
\begin{align*}
B_{n,k}(c_2,2c_3,0,\dots) 
 &= \sum_{\substack{\kappa\le k\\ \nu\le n}} \binom{n}{\nu}B_{\nu,\kappa}(c_2,0,\dots) B_{n-\nu,k-\kappa}(0,2c_3,0,\dots) \\
 &= \sum_{\kappa\le k} \binom{n}{\kappa}c_2^\kappa B_{n-\kappa,k-\kappa}(0,2c_3,0,\dots) \\
 &= \binom{n}{2k-n}c_2^{2k-n} \frac{[2(n-k)]!}{(n-k)!}c_3^{n-k} \\
 &= \frac{n!}{(2k-n)!(n-k)!}\, c_2^{2k-n}c_3^{n-k}.
\end{align*}
Therefore, 
\begin{align*}
  r &= r_0+\sum_{n=0}^\infty \bigg[\sum_{k\ge n/2}^n \frac{(-1)^{n+k+1}}{c_1^k\, n!}  \frac{(n+k)!}{(n+1)!}
   \frac{n!}{(2k-n)!(n-k)!}\, c_2^{2k-n}c_3^{n-k}\bigg] \Big(\frac{c_0}{c_1}\Big)^{n+1},
  \intertext{and with the change $n=2k-j$,} 
  &= r_0+ \frac{c_0}{c_1}\sum_{k=0}^\infty \bigg[\sum_{j=0}^k \frac{(-1)^{k-j+1}}{c_1^k(2k-j+1)}  
  \binom{k}{j}\binom{3k-j}{k} c_2^{j}c_3^{k-j}\bigg] \Big(\frac{c_0}{c_1}\Big)^{2k-j}.
\end{align*}
In summary, if $r_0$ is a simple root mod $p$ of $f(x)=a_0+a_1 x+a_2 x^2 + a_3 x^3\in\Z_p[x]$, then the $p$-adic lift of $r_0$ is given by
\begin{equation}\label{cubicRoot}
  r = r_0-\frac{c_0}{c_1}\sum_{k=0}^\infty \bigg[\sum_{j=0}^k \frac{(-1)^{k-j}c_2^j}{2k-j+1} 
  \binom{k}{j}\binom{3k-j}{k}\Big(\frac{c_0c_3}{c_1}\Big)^{k-j} \bigg] \Big(\frac{c_0}{c_1^2}\Big)^{k},
\end{equation}
where $c_0=f(r_0)$, $c_1=f'(r_0)$, $c_2=f''(r_0)/2$, and $c_3=a_3$. 

\bigskip
Following the same steps as for cubic polynomials, we can obtain the following result.
\begin{proposition} 
Let $p>2$ and $1<\ell<m$. Suppose $f(x)=a_0+a_1 x+a_{\ell} x^{\ell} + a_{m} x^{m} \in \Z_p[x]$ is such that $p\,|\,f(0)$ but $p\!\not|\, f'(0)$. Then $r_0=0$ lifts to a $p$-adic root of $f$ given by
\begin{equation*}
 r=- \frac{a_0}{a_1}\sum_{k=0}^\infty \bigg[\sum_{j=0}^k 
 \frac{(-1)^{m(k-j)+\ell j}\, a_\ell^{j}}{m(k-j)+\ell j-k+1} \binom{k}{j}\binom{m(k-j)+\ell j}{k}
 \Big(\frac{a_0^{m-\ell}a_m}{a_1^{m-\ell}}\Big)^{k-j}\bigg] 
 \Big(\frac{a_0^{\ell-1}}{a_1^\ell}\Big)^{k}. 
\end{equation*}
\end{proposition}

\subsection*{Roots of Unity} 
Let $p>2$ and $f(x)=x^m-1$. Assume that $r_0$ is a single root mod $p$ of $f$. Then $r_0$ lifts to a $p$-adic root of the form \eqref{HenselRoot} with $c_0=r_0^m-1$ and $c_j=\binom{m}{j}r_0^{m-j}$ for $j=1,\dots,m$. Now, since
\begin{equation*}
  j!c_j = r_0^{m-j}(m)_j \;\text{ with }\; (m)_j=\frac{m!}{(m-j)!},
\end{equation*}
homogeneity properties of the Bell polynomials together with identity \eqref{WangEx3.2} give
\begin{align*}
 B_{n+k,k}(1! c_2, 2!c_3,\dots) &= B_{n+k,k}(r_0^{m-1}(m)_1,r_0^{m-2}(m)_2,\dots) \\
 &= r_0^{mk-(n+k)} B_{n+k,k}((m)_1,(m)_2,\dots) \\
 &= r_0^{mk-(n+k)} \frac{1}{k!} \sum_{j=0}^k (-1)^{k-j} \binom{k}{j} (jm)_{n+k}.
\end{align*}
Therefore,
\begin{align*}
\sum_{k=0}^n & \frac{(-1)^{n-k+1}}{c_1^k\,(n+1)!} \binom{2n+1}{n-k} B_{n+k,k}(1!c_1, 2!c_2,\dots) \\
 &= \sum_{k=0}^n \frac{(-1)^{n-k+1}}{(r_0^{m-1}m)^k\,(n+1)!} \binom{2n+1}{n-k} 
  r_0^{mk-(n+k)} \frac{1}{k!} \sum_{j=0}^k (-1)^{k-j} \binom{k}{j} (jm)_{n+k} \\
 &= \sum_{k=0}^n \sum_{j=0}^k\, \frac{(-1)^{n-j+1}}{r_0^{n}\,m^k\,(n+1)!} \binom{2n+1}{n-k} 
  \frac{1}{k!} \binom{k}{j} (jm)_{n+k}.
\end{align*}
Finally, the $p$-adic lift of $r_0$ is given by
\begin{align*}
 r &= r_0 - \frac{c_0}{c_1}\sum_{n=0}^\infty \bigg[\sum_{k=0}^n \sum_{j=0}^k\, \frac{(-1)^{n-j}}{m^k\,(n+1)!} 
 \binom{2n+1}{n-k} \frac{1}{k!} \binom{k}{j} (jm)_{n+k}\bigg] \Big(\frac{c_0}{r_0c_1}\Big)^n,
\end{align*}
where $c_0=f(r_0)$ and $c_1=f'(r_0)$.

\bigskip
For $m=p-1$, the above expression gives an explicit formula for the Teichm\"uller lifts.

\begin{proposition}\label{roots_unity}
Let $p>2$. Every integer $q\in\{1,\dots,p-1\}$ is a $(p-1)$-st root of unity mod $p$ and lifts to a $p$-adic root of unity $\xi_q$ given by
\begin{equation*}
 \xi_q = q-\frac{c_0}{c_1}\sum_{n=0}^\infty \bigg[\sum_{k=0}^n \sum_{j=0}^k\, 
 \frac{(-1)^{n-j}}{(p-1)^k\,(n+1)!} \binom{2n+1}{n-k} \frac{1}{k!} \binom{k}{j} (j(p-1))_{n+k}\bigg] 
 \Big(\frac{c_0}{qc_1}\Big)^{n},
\end{equation*}
where $c_0=q^{p-1}-1$ and $c_1=(p-1)q^{p-2}$. 
\end{proposition}

\section{Factorization of polynomials over $\Z[[x]]$}\label{sec:Factorization}

Let $f(x)=f_0+f_1x+f_2x^2+\dots$ be a formal power series in $\Z[[x]]$. It is easy to prove that $f(x)$ is invertible in $\Z[[x]]$ if and only if $|f_0|=1$. A natural question, initially discussed in \cite{BiGi}, is whether or not a non-invertible element of $\Z[[x]]$ can be factored over $\Z[[x]]$. In recent years, this question has been investigated by several authors, leading to sufficient and in some cases necessary reducibility criteria, see e.g. \cite{Bezivin, BGW2012a, Elliott}. In particular, \cite{Elliott} deals with the factorization of formal power series over principal ideal domains. 

For the case at hand, the following elementary results are known. The formal power series $f(x)=f_0+f_1x+f_2x^2+\dots$ is irreducible in $\Z[[x]]$ if $|f_0|$ is prime, or if $|f_0|=p^w$ with $p$ prime, $w\in\mathbb{N}$, and $\gcd(p,f_1)=1$.

On the other hand, if $f_0$ is neither a unit nor a prime power, then $f(x)$ is reducible. In this case, the factorization algorithm is simple and relies on a recursion and a single diophantine equation, see \cite[Prop.~3.4]{BiGi}. 

Finally, in the remaining case when $f_0$ is a prime power and $f_1$ is divisible by $p$, the reducibility of $f(x)$ in $\Z[[x]]$ is linked to the existence of a $p$-adic root of positive valuation. The goal of this section is to give an explicit factorization over $\Z[[x]]$ for reducible polynomials of the form 
\begin{equation}\label{polynomial}
 f(x)=p^{w}+p^m\gamma_1 x+\gamma_2 x^2+\dots+\gamma_d x^d, \quad m\ge 1,\; w\ge 2, \;d\ge 2,
\end{equation} 
where $\gamma_1,\dots, \gamma_d\in\Z$ and $\gcd(p,\gamma_1)=1$. This is the only type of polynomial for which the reducibility and factorization over $\Z[[x]]$ is not straightforward.

\begin{theorem}\label{FactorizationThm}
Let $p$ be an odd prime and let $f(x)$ be a polynomial of the form \eqref{polynomial}. Assume that $f$ has a simple root $r\in p\Z_p$ with $v_p(r)=\ell\le m$ and $r=p^\ell(1+\sum_{j=1}^\infty e_j p^{\ell j})$ with $e_j\in\Z$. Then $f(x)$ admits the factorization
\begin{equation*}
 f(x) =\bigg(p^\ell-x-x\sum_{n=1}^\infty a_n x^n\bigg)
 \bigg(p^{w-\ell}+(p^{w-2\ell}+p^{m-\ell}\gamma_1) x + x\sum_{n=1}^\infty b_n x^n\bigg),
\end{equation*}
where the coefficients $a_n$ are given by \eqref{Acoeff}, and $b_n=\hat b_n/p^{\ell n}$ with $\hat b_n$ as in \eqref{Bhat_coeff}.
\end{theorem}

\begin{remark}
\begin{enumerate}[(a)]
\item As shown in Lemma~\ref{bhat_divisibility}, $\hat b_n$ is divisible by $p^{\ell n}$, so $b_n\in\Z$ for every $n$.
\item If $r\in p\Z_p$ is a root of $f$ with $v_p(r)=\ell\le m$, then $2\ell\le w$.
\item If $w\le 2m$ and $f$ has a root $r\in p\Z_p$, then $v_p(r)=\ell\le m$ holds. If $w>2m$, then $0$ lifts to a $p$-adic root of $f$, but it is not necessarily true that $f$ has a root of valuation less than or equal to $m$. This property depends on the coefficients $\gamma_2, \gamma_3,\dots$. However, even if that condition fails, $f(x)$ is still reducible and a factorization can be obtained through the algorithm given in \cite[Prop.~2.4]{BGW2012a}.
\item A $p$-adic integer $r$ with $v_p(r)=\ell$ can always be written as $r=p^\ell(e_0+\sum_{j=1}^\infty e_j p^{\ell j})$ with $e_0\in\Z_p^*$. For factorization purposes, we can assume without loss of generality $e_0=1$. Otherwise, consider $g(x)=f(x/e_0^*)$, where $e_0^*$ is such that $e_0e_0^*=1\!\pmod{p^\ell}$.
\item As discussed in \cite{BGW2012a}, the existence of a root in $p\Z_p$ is in many cases (e.g. when $d\le 3$) a necessary condition for the polynomial \eqref{polynomial} to factor over $\Z[[x]]$.
\end{enumerate}
\end{remark}

\begin{remark}
If $f$ has a multiple root in $p\Z_p$, then $f(x)$ admits the simpler factorization
\begin{equation*}
 f(x)= G(x) f_{\text{red}}(x),
\end{equation*}
where $G(x)=\gcd(f(x),f'(x))\in\Z[x]$ and $f_{\text{red}}(x)=f(x)/G(x)$. 
\end{remark}

\subsection*{Proof of Theorem~\ref{FactorizationThm}}
Let $r=p^\ell\big(1+\sum_{j=1}^\infty e_j p^{\ell j}\big)$ be the $p$-adic root of $f$ and define
\begin{equation}\label{Eseries}
\phi(x)=xE(x) \quad\text{with}\quad E(x)=1+\sum_{j=1}^\infty e_j x^j.
\end{equation}
Thus $r=\phi(p^\ell)$ in $\Z_p$ and therefore $p^\ell=\phi^{-1}(r)$. Define $A(x)=p^\ell-\phi^{-1}(x)$. So $A(r)=0$ in $\Z_p$, and by Lemma~\ref{FormalInversion}, we have
\begin{equation*}
 A(x)=p^\ell-\phi^{-1}(x)=p^\ell-x\Big(1+\sum_{n=1}^\infty a_n x^n\Big),
\end{equation*}
where
\begin{equation} \label{Acoeff}
 a_n = \frac{1}{n!}\sum_{k=1}^n (-1)^k \frac{(n+k)!}{(n+1)!}B_{n,k}(1!e_1,2!e_2,\dots) \in\Z.
\end{equation}
Our goal is to find $B(x)\in \Z[[x]]$ such that $f(x)=A(x)B(x)$. For convenience, consider
\begin{equation*}
\hat f(x)=p^{-2\ell}f(p^\ell x) \quad\text{and}\quad \hat A(x) = p^{-\ell}A(p^\ell x).
\end{equation*}
Thus
\begin{equation*}
 \hat A(x) = 1 - x - x\sum_{n=1}^\infty p^{\ell n}a_n x^n.
\end{equation*}

\begin{proposition}
The reciprocal of $\hat A(x)$ is a power series in $\Z[[x]]$ of the form
\begin{equation*}
\hat A(x)^{-1} = \frac{1}{\hat{A}(x)} = 1+ x + x\sum_{n=1}^\infty t_n x^n
\end{equation*}
with
\begin{equation} \label{reciprocalcoeff}
t_n = 1+ \sum_{k=1}^{n} p^{\ell k}\,\frac{n+1-k}{k!}
  \sum_{{j}=1}^{k}(-1)^{j} \frac{(n+{j})!}{(n+1)!} B_{k,{j}}(1!e_1,2! e_2,\dots) \in\Z.
\end{equation}
\end{proposition}
\begin{proof}
By Theorem~B in \cite[Section~3.5]{Comtet}, and using basic properties of partial Bell polynomials, we have
\begin{align*}
\hat A(x)^{-1} 
 &=1+x+\sum_{n=2}^\infty \sum_{k=1}^n \frac{k!}{n!} B_{n,k}(1,2!a_1p^\ell,3!a_2p^{2\ell},\dots)\, x^n \\
 &=1+x+\sum_{n=2}^\infty \sum_{k=1}^n k! \biggl(\sum_{{j}=0}^k\frac{n!}{(n-k)!{j}!}
   B_{n-k,k-{j}}(1! p^{\ell} a_1,2!p^{2\ell}a_2,\dots) \biggr) \frac{x^n}{n!}\\
 &=1+x+\sum_{n=1}^\infty \left(1+\sum_{k=1}^{n} \biggl(\sum_{{j}=1}^{n+1-k}\frac{(n+1-k)!}{k!(n+1-k-{j})!}
   p^{\ell k}B_{k,{j}}(1!a_1,2!a_2,\dots) \biggr) \right) x^{n+1}\\
 &=1+x+x\sum_{n=1}^\infty \left(1+\sum_{k=1}^{n} \frac{p^{\ell k}}{k!}
   \biggl(\sum_{{j}=1}^{k}\frac{(n+1-k)!}{(n+1-k-{j})!} B_{k,{j}}(1!a_1,2!a_2,\dots) \biggr) \right) x^n.
\end{align*}
In the last step, we declare the interior sum to be zero if ${j}> \min(k,n+1-k)$. Thus
\begin{align*}
t_n &=1+\sum_{k=1}^{n} \frac{p^{\ell k}}{k!}\sum_{{j}=1}^{k}\frac{(n+1-k)!}{(n+1-k-{j})!} 
  B_{k,{j}}(1!a_1,2!a_2,\dots) \\
 &=1+\sum_{k=1}^{n} p^{\ell k}\,\frac{n+1-k}{k!}\sum_{{j}=1}^{k}\binom{n-k}{{j}-1}({j}-1)! 
  B_{k,{j}}(1!a_1,2!a_2,\dots)
\end{align*}
Now, if we write $k!a_k$ as
\begin{equation*}
 k!a_k = \sum_{{j}=1}^k \binom{k+{j}}{{j}-1}({j}-1)! B_{k,{j}}(-1!e_1,-2!e_2,\dots),
\end{equation*}
then by means of Theorem~15 in \cite{BGW2012b} we get 
\begin{align*}
\sum_{{j}=1}^{k}\binom{n-k}{{j}-1}({j}-1)! & B_{k,{j}}(1!a_1,2!a_2,\dots) \\
 &= \sum_{{j}=1}^{k} \binom{n+{j}}{{j}-1}({j}-1)! B_{k,{j}}(-1!e_1,-2! e_2,\dots) \\
 &= \sum_{{j}=1}^{k}(-1)^{j} \frac{(n+{j})!}{(n+1)!} B_{k,{j}}(1!e_1,2! e_2,\dots).
\end{align*}
In other words, $t_n$ has the form claimed in \eqref{reciprocalcoeff}.
\end{proof}

Now, motivated by \eqref{reciprocalcoeff}, for $n\ge 1$ we consider
\begin{equation*}
 T_n(x) = 1+ \sum_{k=1}^{\infty}  \frac{n+1-k}{k!}
 \bigg(\sum_{{j}=1}^{k}(-1)^{j} \frac{(n+{j})!}{(n+1)!} B_{k,{j}}(1!e_1,2! e_2,\dots)\bigg) x^k.
\end{equation*}

\begin{lemma} 
With $E(x)$ as in \eqref{Eseries}, we have
\begin{equation*}
 T_{n}(x) = E(x)^{-n-2}\big(E(x)+xE'(x)\big).
\end{equation*}
\end{lemma}
\begin{proof}
Fix $n\ge 1$ and denote
\begin{equation*}
 \tau_k = \frac{1}{k!} \sum_{{j}=1}^{k}(-1)^{j} \frac{(n+{j})!}{n!} B_{k,{j}}(1!e_1,2! e_2,\dots).
\end{equation*}
Then
\begin{equation*}
 T_n(x) = 1+ \sum_{k=1}^{\infty}  \Big(1-\frac{k}{n+1}\Big)\tau_k x^k 
 = 1+ \sum_{k=1}^{\infty} \tau_k x^k  - \frac{1}{n+1}\sum_{k=1}^{\infty}k\tau_k x^k.
\end{equation*}
It is easy to check that $1+\sum_{k=1}^{\infty} \tau_k x^k=E(x)^{-n-1}$. Therefore,
\begin{align*}
 T_n(x) &=E(x)^{-n-1} - \frac{1}{n+1} x \frac{d}{dx}\Big(E(x)^{-n-1}\Big) \\
 &= E(x)^{-n-1} + xE(x)^{-n-2}E'(x) = E(x)^{-n-2}\big(E(x)+xE'(x)\big).
\end{align*}
\end{proof}

As a direct consequence of this lemma we get the recurrence relation
\begin{equation*}
 T_{n-1}(x) = E(x)T_{n}(x),
\end{equation*}
which can be used to define $T_0(x)$ and $T_{-n}(x)$ for $n\ge 1$. More precisely, we let
\begin{equation*}
 T_0(x)=E(x)T_1(x) \;\text{ and }\;
 T_{-n}(x) = E(x)^{n+1} T_1(x) \text{ for } n\ge 1.
\end{equation*}

Given that 
\begin{equation}\label{eq:fhat}
 \hat f(x) = p^{-2\ell}f(p^\ell x) = p^{w-2\ell}+p^{m-\ell}\gamma_1 x + \sum_{n=2}^d p^{\ell(n-2)}\gamma_n x^n,
\end{equation}
the relation $T_{n-j}(x) = E(x)^jT_{n}(x)$ gives
\begin{equation}\label{polyIdentity}
 p^{w-2\ell}T_n(x) + p^{m-\ell}\gamma_1 T_{n-1}(x) + \sum_{j=2}^d p^{\ell(j-2)}\gamma_j T_{n-j}(x) = \hat f(E(x)) T_n(x).
\end{equation}
Moreover, since $E(x)$ is a unit in $\Z[[x]]$, for every $\nu\in\Z$ the function $T_\nu(x)$ is in $\Z[[x]]$ and so $T_\nu(p^\ell)\in\Z_p$. 
 
\begin{lemma}\label{TnCongruences}
For $\nu\ge -1$, the $p$-adic numbers $T_\nu(p^\ell)$ satisfy
\begin{equation*}
 T_\nu(p^\ell)-t_\nu \equiv 0 \!\pmod{p^{\ell(\nu+2)}}
\end{equation*}
with $t_\nu$ as in \eqref{reciprocalcoeff} for $\nu>0$ and $t_0=t_{-1}=1$.
\end{lemma}
\begin{proof}
For $\nu=n\ge1$ the statement is a consequence of the fact that $t_n$ is the $n$-th partial sum of $T_n(p^\ell)$ and the coefficient of $x^{n+1}$ in $T_n(x)$ is zero. Further, given that
\begin{equation*}
E(p^\ell)=1+p^\ell e_1+O(p^{2\ell}) \;\text{ and }\; T_1(p^\ell)=1-p^\ell e_1 + O(p^{2\ell}),
\end{equation*}
we have
\begin{equation*}
T_0(p^\ell) = \big(1+p^\ell e_1+O(p^{2\ell})\big)\big(1-p^\ell e_1 + O(p^{2\ell})\big)\equiv 1 \!\pmod{p^{2\ell}}.
\end{equation*}
This implies $T_0(p^\ell)-t_0\equiv 0 \!\pmod{p^{2\ell}}$.

Finally, since $T_{-1}(p^\ell)=E(p^\ell)^{2}T_1(p^\ell)$, and because $E(p^\ell)^{2}$ and $T_1(p^\ell)$ are both of the form $1+O(p^\ell)$, we get  $T_{-1}(p^\ell)\equiv 1 \!\pmod{p^{\ell}}$, hence $T_{-1}(p^\ell) - t_{-1}\equiv 0 \!\pmod{p^{\ell}}$.
\end{proof}

Using $\hat f(x)$ as in \eqref{eq:fhat}, we now define
\begin{align*}
 \hat B(x) = \hat f(x)\hat A(x)^{-1}
 = \Big(p^{w-2\ell}+p^{m-\ell}\gamma_1 x + \sum_{n=2}^d p^{\ell(n-2)}\gamma_n x^n\Big)
 \Big(1+ x + x\sum_{n=1}^\infty t_n x^n\Big),
\end{align*}
and write it as
\begin{equation*}
 \hat B(x) = p^{w-2\ell}+(p^{w-2\ell}+p^{m-\ell}\gamma_1) x + x\sum_{n=1}^\infty \hat b_n x^n
\end{equation*}
with
\begin{equation} \label{Bhat_coeff}
 \hat b_n = p^{w-2\ell} t_n + p^{m-\ell}\gamma_1 t_{n-1} +\sum_{j=2}^d p^{\ell(j-2)}\gamma_j t_{n-j} \in\Z,
\end{equation}
where $t_n$ is given by \eqref{reciprocalcoeff}, $t_0=t_{-1}=1$, and $t_{-n}=0$ for $n>1$.

\begin{lemma} \label{bhat_divisibility}
The coefficients $\hat b_n$ are divisible by $p^{\ell n}$.
\end{lemma}
\begin{proof}
First of all, since $p^{-\ell}r=E(p^\ell)$ is a $p$-adic root of $\hat f$, identity \eqref{polyIdentity} implies
\begin{equation*}
 p^{w-2\ell}T_n(p^\ell) + p^{m-\ell}\gamma_1 T_{n-1}(p^\ell) + \sum_{j=2}^d p^{\ell(j-2)}\gamma_j T_{n-j}(p^\ell) = 0 \text{ in } \Z_p.
\end{equation*}
Therefore, for $n\ge d-1$,
\begin{align*}
 \hat b_n &= p^{w-2\ell}t_n + p^{m-\ell}\gamma_1 t_{n-1} +\sum_{j=2}^d p^{\ell(j-2)}\gamma_j t_{n-j} \\
 &= p^{w-2\ell}\big(t_n-T_n(p^\ell)\big) + p^{m-\ell}\gamma_1 \big(t_{n-1}-T_{n-1}(p^\ell)\big) +
 \sum_{j=2}^d p^{\ell(j-2)}\gamma_j \big(t_{n-j}-T_{n-j}(p^\ell)\big),
\end{align*}
which by Lemma~\ref{TnCongruences} is congruent to $0$ mod $p^{\ell n}$. Similarly, for $1\le n <d-1$,
\begin{align*}
 \hat b_n &= p^{w-2\ell}t_n + p^{m-\ell}\gamma_1 t_{n-1} +\sum_{j=2}^{n+1} p^{\ell(j-2)}\gamma_j t_{n-j} \\
 &\equiv - \sum_{j=n+2}^d p^{\ell(j-2)}\gamma_j T_{n-j}(p^\ell) \equiv 0 \!\pmod{p^{\ell n}}.
\end{align*}
\end{proof}

\smallskip
Finally, defining $B(x)=p^{\ell}\hat B(x/p^\ell)$, we arrive at the factorization $f(x) = A(x)B(x)$.

\begin{remark}
It is worth mentioning that our method for factorization in $\Z[[x]]$ is not restricted to polynomials and can be applied to power series. As an example, consider 
\begin{equation*}
f(x) =  9+12x+7x^2+8x^3\sum_{k=0}^\infty x^k = 9+12x+7x^2+\frac{8x^3}{1-x},
\end{equation*}
discussed by B\'ezivin in \cite{Bezivin}. This series is reducible in $\Z[[x]]$ and factors as 
\begin{equation*}
  f(x)=\dfrac{(3-x)^2(1+x)}{1-x}.
\end{equation*}
The reader is invited to confirm that the power series version of Theorem~\ref{FactorizationThm} gives the factorization $f(x)=A(x)B(x)$ with $A(x)=3-x$ and $B(x)=\frac{(3-x)(1+x)}{1-x}$. 

An interesting feature of this example is that the partial sums $f_d(x) = 9+12x+7x^2+\cdots$ of $f(x)$ of degree $d\ge 2$ are all irreducible in $\Z[[x]]$. This fact was proved in \cite[Prop.~8.1]{Bezivin}, but it can also be derived from Proposition~3.4 of \cite{BGW2012b} together with the observation that for $d\ge 2$, the polynomial $f_d(x)$ has no roots in $p\Z_p$.
\end{remark}

\section*{Appendix: Some properties of Bell polynomials}\label{sec:appendix}
\setcounter{section}{1}
\setcounter{equation}{0}
\renewcommand{\thesection}{\Alph{section}}

Throughout this paper, we make extensive use of the well-known partial Bell polynomials. For any sequence $x_1,x_2,\dots$, the $(n,k)$-th partial Bell polynomial is defined by
\begin{equation*}
B_{n,k}(x) =\sum_{i\in\pi(n,k)}\frac{n!}{i_1!i_2!\cdots}\left(\frac{x_1}{1!}\right)^{i_1}\left(\frac{x_2}{2!}\right)^{i_2}\cdots,
\end{equation*}
where $\pi(n,k)$ is the set of all sequences $i=(i_1,i_2,\dots)$ of nonnegative integers such that
\begin{equation*}
i_1+i_2+\dots=k \;\text{ and }\; i_1+2i_2+3i_3+\dots=n.
\end{equation*}
Clearly, these polynomials satisfy the homogeneity relation 
\begin{equation*}
B_{n,k}(abx_1,ab^2x_2,ab^3x_3,\dots) = a^kb^n B_{n,k}(x_1,x_2,x_3,\dots).
\end{equation*}

Here are other elementary identities (cf. \cite[Section 3.3]{Comtet}) needed in this paper: 
\begin{gather} 
\label{Comt3l'}
B_{n,k}(\tfrac{x_2}{2},\tfrac{x_3}{3},\dots) = \frac{n!}{(n+k)!} B_{n+k,k}(0,x_2,x_3,\dots), \\
\label{Comt3n}
B_{n,k}(x_1+x_1',x_2+x_2',\dots) = \sum_{\substack{\kappa\le k \\ \nu\le n}} 
\binom{n}{\nu}B_{\nu,\kappa}(x_1,x_2,\dots) B_{n-\nu,k-\kappa}(x_1',x_2',\dots), \\
\label{Comt3n'}
B_{n,k}(0,\dots,0,x_j,0,\dots) = 0, \text{ except } B_{jk,k}=\frac{(jk)!}{k!(j!)^k}\, x_j^k.
\end{gather}
Also of special interest is the identity
\begin{equation}\label{WangEx3.2}
B_{n,k}((a)_1,(a)_2,\dots) = \frac{1}{k!} \sum_{j=0}^k (-1)^{k-j} \binom{k}{j} (ja)_{n},
\end{equation}
where $(a)_n=a(a-1)\cdots(a-n+1)$. This is a special case of \cite[Example~3.2]{WW09}.

For more on Bell polynomials and their applications, see e.g. \cite{Bell,Charalambides,Comtet,WW09}.


\end{document}